\documentclass[a4paper, 10pt]{article}

\usepackage{amsthm,mathtools,amssymb,subcaption,natbib}
\usepackage[shortlabels]{enumitem}
\RequirePackage[colorlinks,citecolor=blue,urlcolor=blue]{hyperref}
\usepackage{geometry}
\newtheorem{Lem}{Lemma}
\newtheorem{As}[Lem]{Assumption}
\newtheorem{Th}[Lem]{Theorem}
\newtheorem{Rm}[Lem]{Remark}
\newtheorem{Def}[Lem]{Definition}
\newtheorem{Notation}[Lem]{Notation}
\newtheorem{Cor}[Lem]{Corollary}

\newcommand{\grad}{\mbox{\rm grad}}
\newcommand{\cov}{\mbox{\rm Cov}}
\newcommand{\acos}{\mbox{\rm ~acos~}}
\newcommand{\vol}{\mbox{\rm vol}}
\newcommand{\fA}{\mathfrak{A}}
\newcommand{\cH}{{\cal H}}
\newcommand{\cD}{{\mathcal{D}}}
\newcommand{\cN}{{\mathcal{N}}}
\newcommand{\EE}{\mathbb E}
\newcommand{\RR}{\mathbb R}
\newcommand{\NN}{\mathbb N}
\newcommand{\SSS}{{{\mathbb S}^m}}
\newcommand{\LLL}{{\mathbb L}^m}
\newcommand{\Prb}{\mathbb P}
\newcommand{\as}{\mbox{\rm ~a.s.}}
\newcommand{\diag}{\mbox{\rm diag}}
\newcommand{\Hess}{\mbox{\rm Hess\,}} 
\DeclareMathOperator*{\argmin}{\mbox{\rm argmin}}
\newcommand{\iid}{\operatorname{\stackrel{i.i.d.}{\sim}}}
\newcommand{\inPrb}{\operatorname{\stackrel{\Prb}{\to}}}
\newcommand{\inas}{\operatorname{\stackrel{a.s.}{\to}}}
\newcommand{\inD}{\operatorname{\stackrel{\cD}{\to}}}
\DeclareMathOperator{\sign}{sign}
\newcommand{\wU}{\widetilde{U}}
\newcommand{\wF}{\widetilde{F}}
\newcommand{\wrho}{\widetilde{\rho}}

\makeatletter
\renewcommand{\paragraph}{%
  \@startsection{paragraph}{4}%
  {\z@}{3.25ex \@plus 1ex \@minus .2ex}{-0.5em}%
  {\normalfont\normalsize\bfseries}%
}
\makeatother

\begin{document}
\title{A Smeary Central Limit Theorem for Manifolds with Application to High Dimensional Spheres}
\author{Benjamin Eltzner\footnote{Felix-Bernstein-Institut f\"ur Mathematische Statistik in den Biowissenschaften, Georg-August-Universit\"at G\"ottingen} 
~and Stephan F. Huckemann$^*$} 
\maketitle

\begin{abstract}
  The (CLT) central limit theorems for generalized Fr\'echet means (data descriptors assuming values in stratified spaces, such as intrinsic means, geodesics, etc.) on manifolds from the literature are only valid if a certain empirical process of Hessians of the Fr\'echet function converges suitably, as in the proof of the prototypical BP-CLT (\cite{BP05}). This is not valid in many realistic scenarios and we provide for a new very general CLT. In particular this includes scenarios where, in a suitable chart, the sample mean fluctuates asymptotically at a scale $n^\alpha$ with exponents $\alpha < 1/2$ with a non-normal distribution. As the BP-CLT yields only fluctuations that are, rescaled with $n^{1/2}$, asymptotically normal, just as the classical CLT for random vectors, these lower rates, somewhat loosely called smeariness, had to date been observed only on the circle (\cite{HH15}). We make the concept of smeariness on manifolds precise, give an example for two-smeariness on spheres of arbitrary dimension, and show that smeariness, although ``almost never'' occurring, may have serious statistical implications on a continuum of sample scenarios nearby. In fact, this effect increases with dimension, striking in particular in high dimension low sample size scenarios.
\end{abstract}

\section{Introduction}

\paragraph{The BP-CLT} The celebrated central limit theorem (CLT) for intrinsic sample means on manifolds by \cite{BP05}, and many subsequent generalizations (e.g. \cite{BB08,H_Procrustes_10,BP13,EllingsonPatrangenaruRuymgaart2013,PatrangenaruEllingson2015, BL17}), rests on a \emph{Taylor expansion} 
\begin{align}\label{eq:BP-taylor}
  \sqrt{n}\,\grad|_{x=x_0} F_n(x) =  \sqrt{n}\,\grad|_{x=0} F_n(x) + \Hess|_{x=\widetilde{x}} F_n(x) \sqrt{n}x_0
\end{align}
(with suitable $\widetilde{x}$ between $0$ and $x_0$) and a generalized \emph{strong law} ($n\to \infty$ and $x_0\to 0$)
\begin{align}\label{eq:BP-Hessian-convergence}
  \Hess|_{x=\widetilde{x}} F_n(x) \inPrb \Hess|_{x=0} F(x)\,.
\end{align}
Here, $X_1,\ldots,X_n\iid X$ is a sample on a smooth manifold $M$, 
\begin{align*}
  F_n(x) = \frac{1}{n}\, \sum_{j=1}^n d(X_j,\phi(x)\big)^2\,,\quad F(x) = \,\EE[d(X,\phi(x))^2]
\end{align*}
are the \emph{sample} and \emph{population  Fr\'echet functions} with a smooth distance $d$ on $M$ and $\phi$ denotes a local smooth chart. By definition, as a minimizer of the sample Fr\'echet function, for the preimage $x_0=x_n$ under $\phi$ of any \emph{sample Fr\'echet mean}, the l.h.s. of  Equation \eqref{eq:BP-taylor} vanishes. If $X$ features a density near the relevant cut loci, Equation \eqref{eq:BP-taylor} is a.s. valid for deterministic points $x_0$ near the preimage $0$ of the \emph{population Fr\'echet mean}, if existent (i.e. if the population Fr\'echet function has a unique minimizer). Further, if the empirical process on the l.h.s. of Equation \eqref{eq:BP-Hessian-convergence}, deterministically indexed in $\widetilde{x}$, is well defined, and not only a.s. well defined, the convergence in Equation \eqref{eq:BP-Hessian-convergence} is valid also for $x_0=x_n$, and since the properly rescaled sum of i.i.d. random variables  $\sqrt{n}\,\grad|_{x=0} F_n(x)$ converges to a Gaussian, this strain of argument gives the BP-CLT
\begin{align}\label{eq:BP-CLT} 
  \sqrt{n}  x_n \inD \cN(0,\Sigma)\,,
\end{align}
with suitable covariance matrix $\Sigma$, if the Hessian on the r.h.s of Equation \eqref{eq:BP-Hessian-convergence} is invertible. 
  
\paragraph{Beyond the BP-CLT} Recently in \citet[Example 1]{HH15}, an example on the circle with log coordinates $x\in [-\pi,\pi)$ has been provided, with population Fr\'echet mean at $x=0$ and a local density $f$ near the antipodal $-\pi$. For $x>0$ sufficiently small, the rescaled sample Fr\'echet function takes the value
\begin{align*}
  n F_n(x) &= \sum_{x -\pi \leq X_j} (X_j -x)^2 + \sum_{X_j <x-\pi} (X_j+2\pi - x)^2\\
  &= \sum_{j=1}^n(X_j - x)^2 +  4\pi \sum_{X_j < x-\pi} (X_j - x + \pi)
\end{align*}
so that 
the l.h.s. of  Equation \eqref{eq:BP-Hessian-convergence} is only a.s. well defined with value $\Hess|_x F_n(x) = 2$ a.s. (as in the Euclidean case). The r.h.s., however, assume the value $\Hess|_{x=0} F(x) = 2- 4\pi f(-\pi)$. Hence, in case of $f(-\pi)\neq 0$, the convergence \eqref{eq:BP-Hessian-convergence} is no longer valid, making the above strain of argument no longer viable. Still, as shown in \cite{MQC12,HH15}, as long as $2\pi f(-\pi)<1$, the BP-CLT \eqref{eq:BP-CLT} remains valid.

Further, in \cite{HH15} it was shown that $1=2\pi f(-\pi)$ is possible, so that the BP-CLT \eqref{eq:BP-CLT}, which, under square integrability, holds universally for Euclidean spaces, is wrong for such 2D vectors confined to a circle, by giving examples in which the fluctuations may asymptotically scale with $n^\alpha$ with exponents $\alpha$ strictly lower than one-half.

This new phenomenon has, somewhat loosely, been called \emph{smeariness}, it can only manifest in a non-Euclidean geometry. Examples beyond the circle were not known to date.

\paragraph{A General CLT} Making the concept of smeariness on manifolds precise, using Donsker Theory (e.g. from \cite{vanderVaar2000astat}) and avoiding the sample Taylor expansion \eqref{eq:BP-taylor} as well as the non generally valid convergence condition \eqref{eq:BP-Hessian-convergence}, we provide for a general CLT on manifolds that requires no assumptions other than a unique population mean and a sufficiently well behaved distance. With the \emph{degree of smeariness} $\kappa\geq 0$ our general CLT takes the form
\begin{align}\label{eq:EH-CLT}
  \sqrt{n}  x^{\kappa+1}_n \inD \cN(0,\Sigma)\,,
\end{align}
where $ x_n^{\kappa+1}$ is defined componentwise. Then,  $ x_n$ scales with $n^\alpha$, $\alpha= {\frac{1}{2(\kappa +1)}}$, and $\kappa =0$ corresponds to the usual CLT valid on Euclidean spaces, and to the BP-CLT \eqref{eq:BP-CLT}. 

We phrase our general CLT in terms of sufficiently well behaved generalized Fr\'echet means, e.g. geodesic principal components (\cite{HZ06,HHM07}) or principal nested spheres (\cite{Jung2010,JFM2011}). While we discuss some intricacies in Remark \ref{rmk:hurdle}, their details are beyond the scope of this paper and left for future research. In general, generalized Fr\'echet means are random object descriptors (e.g. \cite{MarronAlonso2014}) that take values in a stratified space and for our general CLT we require only 
\begin{enumerate}[(i)]
  \item a law of large numbers for a unique generalized Fr\'echet mean $\mu$, 
  \item a local chart at $\mu$, sufficiently smooth,
  \item an a.s. Lipschitz condition and an a.s. differentiable distance between $\mu$ and data, and
  \item a population Fr\'echet function, sufficiently smooth at $\mu$.  
\end{enumerate}

Further, we give an example for two-smeariness on spheres of arbitrary dimension, and show that smeariness, although ``almost never'' occurring, may have serious statistical implications on a continuum of sample scenarios nearby. Remarkably, 
this effect increases with dimension, striking in particular in high dimension low sample size scenarios.    

\section{A General Central Limit Theorem}\label{CLT:scn}

In a typical scenario of non-Euclidean statistics, a two-sample test is applied to two groups of manifold-valued data or more generally to data on a manifold-stratified space. Such a test can be based on certain data descriptors such as intrinsic means (e.g. \cite{BP05,MPPPR07,PatrangenaruEllingson2015}), best approximating geodesics (e.g. \cite{H_ziez_geod_10}), best approximating subspaces within a given family of subspaces and entire flags thereof (cf. \cite{HuckemannEltzner2017}), and asymptotic confidence regions can be constructed from a suitable CLT for such descriptors. In this section we first introduce the setting of generalized Fr\'echet means along with standard assumptions, we then recollect and expand some Donsker Theory from \cite{vanderVaar2000astat} and state and prove our general CLT.

\subsection{Generalized Fr\'echet Means and Assumptions}

Fr\'echet functions and Fr\'echet means have been first introduced by \cite{F48} for squared metrics $\widetilde{\rho}:Q\times Q\to [0,\infty)$ on a topological space $Q$ and later extended to squared quasimetrics by \cite{Z77}. Generalized Fr\'echet means as follows have been introduced by \cite{H_ziez_geod_10}. A simple setting is given when $P=Q$ is a Riemannian manifold and $\widetilde{\rho} = d^2$ is the squared geodesic intrinsic distance. Then a generalized Fr\'echet mean is a minimizer with respect to squared distance, often called a \emph{barycenter}.

\begin{Notation}\label{notation}
	Let $P$ and $Q$ be separable topological spaces, $Q$ is called the \emph{data space} and $P$ is called the \emph{descriptor space}, linked by a continuous map $\widetilde{\rho}: P\times Q \to [0,\infty)$ reflecting distance between a data descriptor $p\in P$ and a datum $q\in Q$. Further, with a silently underlying probability space $(\Omega,\fA,\Prb)$, let $X_1,\ldots,X_n\iid X$ be random elements on $Q$, i.e. they are Borel-measurable mappings $\Omega \to Q$. They give rise to
	\emph{generalized population} and \emph{generalized sample Fr\'echet functions},
  \begin{align*}
    \widetilde{F} &: p \mapsto \EE [\widetilde{\rho}(p, X)]\,, & \widetilde{F}_n &: p \mapsto \frac{1}{n} \sum_{j=1}^n \limits \widetilde{\rho}(p, X_j)]\,,
  \end{align*}
  respectively, and their \emph{generalized population} and \emph{generalized sample Fr\'echet means}
  \begin{align*}
    \widetilde{E} &= \left\{ p \in P : \widetilde{F}(p) = \inf_{p \in P} \limits \widetilde{F}(p) \right\}\,, & \widetilde{E}_n &= \left\{ p \in P : \widetilde{F}_n(p) = \inf_{p \in P} \limits \widetilde{F}_n(p) \right\} \,,
  \end{align*}
  respectively. Here the former set is empty if the expected value is never finite. 
\end{Notation}

With Assumption 2.3 further down, $P$ is a manifold locally near $\mu$, so that convergence in probability in the following assumption is well defined.

\begin{As}[Unique Mean with Law of Large Numbers] \label{as:unique}
  In fact, we assume that $\widetilde{E}$ is not empty but contains a single descriptor $\mu \in P$ and that for every measurable selection $\mu_n \in \widetilde{E}$,
  \begin{align*}
    \mu_n \inPrb \mu\,.
  \end{align*}
\end{As}

\begin{As}[Local Manifold Structure] \label{as:local-manifold}
  With $2\leq r\in \NN$ assume that there is a neighborhood $\widetilde{U}$ of $\mu$ that is an $m$-dimensional Riemannian manifold, $m\in \NN$, such that with a neighborhood $U$ of the origin in $\RR^m$ the exponential map $\exp_\mu : U \to \widetilde{U}$, $\exp_\mu(0)= \mu$, is a $C^r$-diffeomorphism, and we set for $p = \exp_\mu(x), p' = \exp_\mu(x')\in \widetilde{U}$ and $q\in Q$,
  \begin{align*}
    \rho &: (x,q) \mapsto \widetilde{\rho} (\exp_\mu(x), q)\,,\\
    F &: x \mapsto \widetilde{F}(\exp_\mu(x))\,, &
    F_n &: x \mapsto \widetilde{F}_n(\exp_\mu(x))\,.
  \end{align*}
  It will be convenient to extend $F_n$ to all of $\RR^m$ via $F_n(x) = F_n(0)$ for $x\in \RR^m\setminus U$.
\end{As}

\begin{As}[Almost Surely Locally Lipschitz and Differentiable at Mean]\label{as:Lipschitz}
  Further assume that 
  \begin{enumerate}[(i)]
    \item the gradient $\dot{\rho}_0 (X) := \grad_x \rho(x,X)|_{x=0}$ exists almost surely;
    \item there is a measurable function $\dot{\rho}: Q \to \RR$ satisfying $\EE [\dot{\rho}(X)^2] < \infty$ for all $x\in U$ and that the following Lipschitz condition 
    \begin{align*}
      |\rho(x_1, X) - \rho(x_2, X)| \le \dot{\rho}(X) \| x_1 - x_2 \|\as
    \end{align*}
    holds for all $x_1,x_2\in U$.
  \end{enumerate}
\end{As}

\begin{As}[Smooth Fr\'echet Function]\label{as:Taylor}
  With $2\leq r\in \NN$ and a non-vanishing tensor $T=(T_{j_1,\ldots,j_r})_{1\leq j_1\leq\ldots\leq j_r\leq m}$, assume that the Fr\'echet function admits the power series expansion
  \begin{align}
    F(x) &= F(0) + \sum_{1\leq j_1\leq\ldots\leq j_r\leq m} x_{j_1}\ldots x_{j_r} T_{j_1,\ldots,j_r} + o(\|x\|^{r})\,.  \label{eq:power_series_full}
  \end{align}
\end{As}

The tensor in $T$ in (\ref{eq:power_series_full}) can be very complicated. As is well known, for $r=2$, every symmetric tensor is diagonalizable ($m(m+1)/2$ parameters involved), which is, however, not true in general. For simplicity of argument, however, we assume that $T$ is diagonalizable with non-zero diagonal elements so that Assumption \ref{as:Taylor} rewrites as follows. In this formulation, we can also drop our assumption that $r\in \NN$.
    
\begin{As}\label{as:Taylor-2nd}
  With $2\leq r\in \RR$, a rotation matrix $R\in SO(m)$ and $T_1,\ldots,T_m \neq 0$ assume that the Fr\'echet function admits the power series expansion 
  \begin{align}
    F(x) &= F(0) + \sum_{j=1}^m \limits T_{j} |(Rx)_{j}|^r + o(\|x\|^{r})\,.\label{eq:power_series_full-2nd}
  \end{align}
\end{As}

\begin{Rm}[Typical Scenarios] Let us briefly recall typical scenarios. In many applications, $Q$ is 
\begin{enumerate}[(a)]
  \item globally a complete smooth Riemannian manifold, e.g. a sphere (cf. \cite{MJ00} for directional data), a real or complex projective space (cf. \cite{K84,MP05} for certain shape spaces) or the space of positive definite matrices (cf. \cite{DKZ09} for diffusion tensors),
  \item a non-manifold \emph{shape space} which is a quotient of a Riemannian manifold under an isometric group action with varying dimensions of isotropy groups (e.g..\cite{DM98,KBCL99}, for spaces of three- and higher-dimensional shapes), 
  \item a general stratified space where all strata are manifolds with compatible Riemannian structures, e.g. phylogenetic tree spaces (cf. \cite{BilleraHolmesVogtmann2001,MoultonSteel2004}, for varying geometries).
\end{enumerate}

\noindent On these spaces,
\begin{itemize}
  \item[($\alpha$)] in most of the above applications, $P=Q$ and \emph{intrinsic means} are considered where $\wrho$ is the squared  geodesic distance induced from the Riemannian structure.
  \item[($\beta$)] In other examples, $P= \Gamma$, the space of geodesics on $Q$ is considered, in view of PCA-like dimension reduction methods (e.g. \cite{fletch4,HZ06,HHM07}), or
  \item[($\gamma$)] $P$ is a family of subspaces of $Q$, or even a space of \emph{nested subspheres} in \cite{Jung2010,JFM2011}; more general families have been recently considered in generic dimension reduction methods, e.g. \cite{Sommer2016,Pennec2017}.    
\end{itemize}
\end{Rm}

\begin{Rm}\label{rmk:hurdle} Of the above assumptions some are harder to prove in real examples than others.
  \begin{enumerate}[(i)]
    \item Of all above assumptions, uniqueness (first part of Assumption \ref{as:unique}) seems most challenging to verify. To date, only for intrinsic means on the circle the entire picture is known, cf. \citet[p. 182 ff.]{HH15}. For complete Riemannian manifolds, uniqueness for intrinsic means has been shown if the support is sufficiently concentrated (cf. \cite{Ka77, KWS90,L01,Gr05,Afsari10}) and intrinsic sample means are unique a.s. if from a distribution absolutely continous w.r.t. Riemannian measure, cf. \citet[Remark 2.6]{BP03} (for the circle) and \citet[Theorem 2.1]{ArnaudonMiclo2014} (in general).
    \item For the above typical scenarios, we anticipate that the other assumptions are often valid in concrete applications. 
      
    For instance, Assumption \ref{as:local-manifold} is also true on non-manifold shape spaces, due to the \emph{manifold stability theorem} \citet[Corollary 1]{H_meansmeans_12}. It may be not be valid, however, on arbitrary stratified spaces, cf. \cite{HHMMN13,H_Mattingly_Miller_Nolen2015}.
    \item Moreover, Assumption \ref{as:Lipschitz} is only slightly stronger than \emph{uniform coercivity} (condition (2) in \citet[p. 1118]{H_ziez_geod_10}) which suffices for the strong law (second part of  Assumption \ref{as:unique}), cf. \citet[Theorem A4]{H_ziez_geod_10} and \citet[Theorem 4.1]{HuckemannEltzner2017}, and this has been established for principal nested spheres  in \citet[Theorem 3.8]{HuckemannEltzner2017} and for geodesics with nested mean on Kendall's shape spaces in \citet[Theorem 3.9]{HuckemannEltzner2017}. In consequence of Lemma \ref{lem3} below, we have that Assumption \ref{as:Lipschitz} holds for intrinsic means of distributions on spheres which feature a density near the antipodal of the intrinsic population mean.
  \end{enumerate}
  A more detailed analysis is beyond the scope of this paper and left for future research. 
\end{Rm}

\subsection{General CLT}

For the following, fix a measurable selection $\mu_n\in \widetilde{E}_n$. Due to $\mu_n\inPrb \mu$ from Assumption \ref{as:unique}, we have $\Prb\{\mu_n \in \widetilde{U}\} \to 1$, and in accordance with the convention in Assumption \ref{as:local-manifold}, setting
\begin{align*}
  x_n := \left\{\begin{array}{cl}
  \exp^{-1}_{\mu}(\mu_n)&\mbox{ if }\mu_n  \in \widetilde{U} \\ 0&\mbox{ else}
  \end{array}\right.\,,
\end{align*}
note that
\begin{align}\label{eq:approx-F-fcn}
  F_n(0)\geq F_n(x_n) =  \widetilde{F}_n(\mu_n) + o_p(1)\,,
\end{align}
because $\Prb\{F_n(x_n) - \widetilde{F}_n(\mu_n) >\epsilon\} = \Prb\{\mu_n \not \in \widetilde{U}\} \to 0$ for all $\epsilon >0$.

The following is a direct consequence of \citet[Lemma 5.52]{vanderVaar2000astat}, replacing maxima with minima, where, due to continuity of $\widetilde{\rho}$, we have no need for outer measure and outer expectation, and, due to our setup, no need for approximate minimizers.

\begin{Lem}\label{lem:powers}
  Assume that for fixed constants $C$ and $\alpha > \beta$ for every $n$ and for sufficiently small $\delta$
  \begin{align}
    \sup_{\|x\| < \delta} \limits \left| F(x) - F(0) \right| &\le C \delta^\alpha \label{eq:Taylor_order}\,,\\
    \EE\left[ n^{1/2} \sup_{\|x\| < \delta} \limits \big| F_n(x) - F(x) - F_n(0) + F(0) \big| \right] &\le C \delta^\beta\,. \label{eq:outer_exp}
  \end{align}
  Then, any a random sequence $\RR^m\ni y_n \inPrb 0$ that satisfies $F_n(y_n) \le F_n(0)$ also satisfies $n^{1/(2 \alpha - 2 \beta)} y_n = \mathcal{O}_P(1)$.
\end{Lem}

As a first step, the following generalization of \citet[Corollary 5.53, only treating the case $r=2$]{vanderVaar2000astat} gives a bound for the scaling rate in the general CLT, so that also in case of $r\geq 2$, $\sqrt{n}x_n=o_p(1)$. 

\begin{Cor}\label{cor:asymptotic_rate}
  Under Assumptions \ref{as:unique}, \ref{as:local-manifold} and \ref{as:Lipschitz}, as well as Assumption \ref{as:Taylor} or \ref{as:Taylor-2nd},
  \begin{align*}
    n^{1/(2r - 2)} x_n = \mathcal{O}_P(1)\,.
  \end{align*}
\end{Cor}

\begin{proof}
  By  Assumption \ref{as:unique} and definition, $x_n \inPrb 0$ with $F_n(x_n)\leq F_n(0)$, cf. \eqref{eq:approx-F-fcn}. Hence, Lemma \ref{lem:powers} yields the assertion, because 
  for $\alpha = r$, \eqref{eq:Taylor_order} follows at once from \eqref{eq:power_series_full} or from \eqref{eq:power_series_full-2nd}, and under Assumption \ref{as:Lipschitz}, \eqref{eq:outer_exp}, for $\beta=1$ follows word by word from the proof of  \citet[Corollary 5.53]{vanderVaar2000astat}.
\end{proof}

As the second step, the following Theorem, which is a generalization and adaption of  \citet[Theorem 5.23]{vanderVaar2000astat}, shows that under Assumption \ref{as:Taylor-2nd} the above bound gives the exact scaling rate, including the explicit limiting distribution.

\begin{Th}[General CLT for Generalized Fr\'echet Means]\label{thm:CLT}
  Under Assumptions \ref{as:unique},  \ref{as:local-manifold}, \ref{as:Lipschitz} and \ref{as:Taylor-2nd}, we have
  \begin{align*}
   &n^{1/2} \left((Rx_n)_1 |(Rx_n)_1|^{r-2},\ldots, (Rx_n)_m |(Rx_n)_m|^{r-2}\right)^T\\
   \inD \quad & \cN\left(0,\frac{1}{r^2} T^{-1} \cov[\grad|_{x=0}\rho(x,X)] T^{-1}\right)\,
  \end{align*}
  with $T=\diag(T_1,\ldots,T_m)$. 
  In particular for every coordinate $j=1,\ldots,m$, 
  \begin{align*}
    n^{\frac{1}{2r-2}} (R^T x_n)_j \inD \cH_j
  \end{align*}
  where $(\sign(\cH_1)\cH_1^{r-1}, \ldots,\sign(\cH_1)\cH_m^{r-1}) $ has the above multivariate Gaussian limiting distribution.
\end{Th}

\begin{proof}
  For $z\in U$ and $2(r-1) = 1/s$, let us abbreviate
  \begin{align*}
    \tau_n(z,X) &:= n^s(\rho(z n^{-s}, X) - \rho(0, X)) - z^T \dot{\rho}_0 (X)\\
    G_n &:= n^{1/2}\left(\frac{1}{n} \sum_{j=1}^n \dot{\rho}_0 (X_j)  - \EE \left[\dot{\rho}_0 (X)\right] \right)\,,
  \end{align*}
  where we set $\rho(z n^{-s}, X)=\rho(0, X)$ if $z n^{-s}\not \in U$. Then, due to Assumptions \ref{as:Lipschitz} and \ref{as:Taylor-2nd}, and $1/2 + s - sr =0$,
  \begin{align*}
    &n^{1/2} \left( \frac{1}{n} \sum_{j=1}^n \limits \left( \tau_n(z,X_j) \right) - \EE \left[ \tau_n(z,X) \right] \right)\\
    =& n^{1/2+s}\Big(F_n(zn^{-s}) - F_n(0) - F(zn^{-s}) + F(0)\Big) - z^T G_n\\
    =& n^{1/2+s} \Big(F_n(zn^{-s}) - F_n(0)\Big) - \sum_{j=1}^m \limits T_{j} |(Rz)_{j}|^r - z^T G_n + o(\|z\|^r) 
  \end{align*}
  is a sequence of stochastic processes, indexed in $z\in U$, with zero expectation and variance converging to zero. By argument from the proof of  \citet[Lemma 19.31]{vanderVaar2000astat}, due to Assumption \ref{as:Lipschitz}, $z$ can be replaced with any random sequence $z_n = O_p(1)$, cf. also the proof of \citet[Lemma 5.23]{vanderVaar2000astat} for $r=2$, yielding, 
  \begin{align}\label{eq:master-equation}
    n^{1/2+s}\Big(F_n(z_n n^{-s}) - F_n(0)\Big) = \sum_{j=1}^m \limits T_{j} |(Rz_n)_{j}|^r + z_n^T G_n + o_P(1) \, .
  \end{align}
  By Corollary \ref{cor:asymptotic_rate}, $z_n = x_n n^s$ is a valid choice in equation \eqref{eq:master-equation}. Comparison with any other $z_n = O_P(1)$, because $\mu_n$ is a minimizer for $\widetilde{F}_n$ and $F_n(x_n)$ deviates only up to $o_p(1)$ from $\widetilde{F}_n(\mu_n)$, due to \eqref{eq:approx-F-fcn}, reveals,
  \begin{align*}
    n^{1/2+s}\Big(F_n(x_n) - F_n(0)\Big)  \le n^{1/2+s}\Big(F_n(z_n n^{-s}) - F_n(0)\Big) + o_P(1)\,.
  \end{align*}
  This asserts that $Rx_nn^s$ is a minimizer, up to $o_P(1)$, of the right hand side of \eqref{eq:master-equation}, i.e. of
    \begin{align*}
	f:w &\mapsto f(w):=\sum_{j=1}^m \limits T_{j} |(w)_j|^{r} + w^TRG_n \,.
  \end{align*}
  This function, however, has a unique minimizer, given on the component level ($j=1,\ldots,m$) by 
  \begin{align*}
    r T_j \sign((w_n)_j) |(w_n)_j|^{r-1} = - (R G_n)_j \quad \mbox{ i.e. } \quad (w_n)_j |(w_n)_j|^{r-2} = - \frac{(R G_n)_j}{r T_j} \,,
  \end{align*}
  yielding
  \begin{align*}
  \sqrt{n} (Rx_n)_j |(Rx_n)_j|^{r-2} & = - \frac{(RG_n)_j}{rT_j} +o_P(1)\,.
  \end{align*}
  Now the classical CLT gives the first assertion. The second also follows from the above display, since for $z= (Rx_n)_j$ and $H= -(RG_n)_j/rT_j$, the equation $\sqrt{n} \sign(z) |z|^{r-1} = H$ implies $\sign(z) = \sign(H)$ and hence
  \begin{align*}
    n^{\frac{1}{2r-2}} z = n^{\frac{1}{2r-2}} \sign(z) |z| = \sign(H) |H|^{\frac{1}{r-1}}\,.
  \end{align*}

\end{proof}

\begin{Rm} 
	The above arguments rely among others on the fact that due to Assumption  \ref{as:Lipschitz}, a specific convergence, different from \eqref{eq:BP-Hessian-convergence}, that can be easily verified for empirical processes indexed in a deterministic bounded variable, are also valid if the index varies randomly, bounded in probability. This can be weakened to the  requirement, that the function class $\rho(x, \cdot)$ possesses the \emph{Donsker} property, cf. \citet[Chapter 19]{vanderVaar2000astat}.
\end{Rm}

\section{Smeariness}\label{scn:smeary}

Recall from \cite{H_Semi_Intrinsics_15} that a sequence of random vectors $X_n$ is \emph{$k$-th order smeary} if
$n^{\frac{1}{2(k+1)}} X_n$ has a non-trivial limiting distribution as $n\to \infty$. 

With this notion, the classical central limit theorem in particular asserts for random vectors with existing second moments that the fluctuation of sample means around the population mean is $0$-th order smeary, also called \emph{nonsmeary}.

It has been shown in \cite{HH15} that the fluctuation of random directions on the circle of sample means around the population mean may feature smeariness of any given positive integer order. It has been unknown to date, however, whether the  phenomenon of smeariness extends to higher dimensions, in particular, to positive curvature. 

To this end, we now make the concept of smeariness on manifolds precise.

\begin{Def} Let $(\Omega,\fA,\Prb)$ be a probability space, $X:\Omega\to \RR^m$ a non-deterministic random vector and $k>-1$. Then a sequence of Borel measurable mappings $X_n:\Omega_n \to \RR^m$ ($n\in \NN$) with $\Omega_n \in \fA$, $\Prb(\Omega_n)\to 1$ ($n\to \infty$) is \emph{$k$-smeary with limiting distribution} of $X$ if 
\begin{align*}
  \Prb\left\{n^{\frac{1}{2(k+1)}}X_n \in B|\Omega_n\right\} &\to \Prb\{X\in B\}\mbox{ as }n\to \infty\mbox{ for all Borel sets }B\subset \RR^m\,.
\end{align*}
In this case we write $n^{\frac{1}{2(k+1)}}X_n \inD X$.
\end{Def}

Note that $-1<k$-smeariness implies that $\Prb\{X_n \in B|\Omega_n\} \to 1_{0\in B}$ for all Borel $B\subset \RR^m$. As usual, we abbreviate this with $X_n \inPrb 0$.

\begin{Lem} \label{lem:chart_independence}
  Let $X_n: \Omega_n \to \RR^m$ be Borel measurable with $\Prb(\Omega_n)\to 1$ and $X_n\inPrb 0$, consider a continuously differentiable local bijection $\Phi:U\to V$  preserving the origin $0\in U,V$ open $\subset \RR^m$, set $Y_n = \Phi(X_n): \Omega_n \cap \{X_n \in U\} \to \RR^m$ and let $k>-1$. Then 
  \begin{center}
    $X_n$ is $k$-smeary $\Leftrightarrow$ $Y_n$ is $k$-smeary.
  \end{center}
  In particular, if $X$ has the limiting distribution of $n^{\frac{1}{2(k+1)}}X_n$, then $D\Phi(0) X$ has the limiting distribution of $n^{\frac{1}{2(k+1)}}Y_n$. Here $D\Phi(x)$ denotes the differential of $\Phi$ at $x\in U$ and $\det\big(D\Phi(0)\big) \neq 0$ due to invertibility of $\Phi$.
\end{Lem}

\begin{proof}
The implication ``$\Rightarrow$'' is a direct consequence of a Taylor expansion and the continuity theorem with a suitable point  $\widetilde{X}_n\inPrb 0$ between the origin and $X_n$ as follows
\begin{align*}
  \Prb\left\{n^{\frac{1}{2(k+1)}}Y_n \in B|\Omega_n\cap \{X_n \in U\}\right\} &=\Prb\left\{n^{\frac{1}{2(k+1)}}D\Phi(\widetilde{X}_n) X_n \in B|\Omega_n\cap \{X_n \in U\}\right\}\\
  &\to \Prb\left\{D\Phi(0) X\in B\right\}
\end{align*}
because  $\Prb\{X_n \in U\}\to 1$ due to  $X_n\inPrb 0$.

Similarly, the implication ``$\Leftarrow$'' follows. Suppose that $Y$ has the limiting distribution of $n^{\frac{1}{2(k+1)}}Y_n$. Then
\begin{align*}
\Prb\left\{n^{\frac{1}{2(k+1)}}X_n \in B|\Omega_n\right\} &= \Prb\left\{n^{\frac{1}{2(k+1)}}D\Phi(\widetilde{X}_n)^{-1} Y_n \in B|\Omega_n\cap \{X_n \in U\}\right\} \\ &\hspace*{2cm} + \Prb\left\{n^{\frac{1}{2(k+1)}}X_n \in B|\Omega_n \cap \{X_n \not\in U\}\right\}\\
&\to \Prb\left\{D\Phi(0)^{-1} Y\in B\right\}\,,
\end{align*}
again due to the hypothesis   $X_n\inPrb 0$.
\end{proof}

In consequence of Lemma \ref{lem:chart_independence}, we have the following general definition.

\begin{Def}
  A sequence $\mu_n\inPrb \mu$ of random variables on a $m$-dimensional manifold $M$ is \emph{$k$-smeary} if in one -- and hence in every -- continuously differentiable chart $\phi^{-1}:\wU \to \RR^m$ around $\mu \in \wU \subset M$ the sequence of vectors $\phi^{-1}(\mu_n)-\phi^{-1}(\mu) : \{\mu_n \in \wU\} \to \RR^m$ is $k$-smeary.  
\end{Def}

\begin{Rm}
  In particular, the order of smeariness is independent of the chart chosen.
\end{Rm}

\section{An Example of Two-Smeariness on Spheres} 
\subsection{Setup}\label{setup:scn}

Consider a random variable $X$ distributed on the $m$-dimensional unit sphere $\SSS$ ($m\geq 2$) 
that is uniformly distributed on the lower half sphere 
$ \LLL = \{q\in \SSS: q_2\leq 0\}$ with total mass $0<\alpha<1$ and assuming the \emph{north pole} $\mu=(0,1,0,\ldots,0)^T$ with probability $1-\alpha$. Then we have the \emph{Fr\'echet function}
\begin{align*}
  \wF:\SSS \to [0,\infty),~p\mapsto  \int_\SSS \wrho(p,q) \,d\Prb^X(q)
\end{align*}
involving the \emph{squared spherical distance} $\wrho(p,q) = \arccos\langle p,q\rangle^2$ based on the standard inner product $\langle\cdot,\cdot\rangle$ of $\RR^{m+1}$. 
Every minimizer $p^* \in \SSS$ of $F$ is called an \emph{intrinsic Fr\'echet population mean} of $X$. 

With the volume of $\SSS$ given by
\begin{align*}
  v_m = {\vol}(\SSS)= \frac{2\pi^{\frac{m+1}{2}}}{\Gamma\left(\frac{m+1}{2}\right)}
\end{align*}
define
\begin{align*}
  \gamma_m = \frac{v_{m+1}}{2v_m} = \frac{\sqrt{\pi}}{2} \,\frac{\Gamma\left(\frac{m+1}{2}\right)}{\Gamma\left(\frac{m+2}{2}\right)} \,.
\end{align*}
Moreover, we have the \emph{exponential chart} centered at $\mu\in \SSS$ with inverse
\begin{align*}
  \exp_\mu^{-1}(p) = (e_1,e_3,\ldots,e_{m+1})^T \big(p -\langle p,\mu\rangle\mu\big) \frac{\arccos\langle p,\mu\rangle}{\| p -\langle p,\mu\rangle\mu\|} = x\in \RR^m\, 
\end{align*}
where $e_1,\ldots,e_{m+1}$ are the standard unit column vectors in $\RR^{m+1}$. Note that $\exp_\mu^{-1}$ has continuous derivatives of any order in $\wU=\SSS\setminus\{-\mu\}$ and recall that $e_2=\mu$.

\subsection{Derivatives of the Fr\'echet Function}\label{scn:derivatives}

\begin{Lem}\label{lem:crescent}
  With the above notation, the function $F= \wF\circ \exp_\mu$ has derivatives of any order for $x\in \exp_\mu^{-1}(\wU)$ with $\|x\| < \pi/2$. For
  $\alpha = 1/(1+\gamma_m)$ the north pole $\mu$ gives the unique intrinsic Fr\'echet mean with $\Hess|_{x=0} F\circ \exp_\mu(x) = 0$. Moreover, for any choice of $0<\alpha<1$,
  \begin{align*}
    \partial_i\partial_k\partial_l|_{x=0} F&=0\\
    \partial_i\partial_k\partial_l\partial_s|_{x=0} F&=c_m \delta_{i,k,l,s}
  \end{align*}
  for every $1\leq i,k,l,s\leq m$ with the constant $c_m=\frac{2\gamma_m}{1+\gamma_m}\,\frac{m-1}{m+2} >0$.
\end{Lem}

\begin{proof} 
  For convenience we choose polar coordinates $\theta_1,\ldots,\theta_{m-1} \in [-\pi/2, \pi/2]$ and $\phi \in [-\pi, \pi)$ in the non-standard way
  \begin{align*}
    q=\begin{pmatrix}q_1\\q_2\\\vdots\\q_{m-1}\\q_{m}\\q_{m+1}\end{pmatrix} = \begin{pmatrix} - \left(\prod_{j=1}^{m-1} \cos \theta_j\right)  \cos  \phi \\   - \left(\prod_{j=1}^{m-1} \cos \theta_j\right)  \sin  \phi \\ \vdots \\- \cos \theta_1\,\cos\theta_2\,\sin\theta_3\\
    - \cos \theta_1\,\sin\theta_2\ \\ \sin\theta_1 \end{pmatrix} \,,
  \end{align*}
  such that the north pole $\mu$ has coordinates $(0,\ldots,0,-\pi/2)$. In fact, we have chosen these coordinates so that w.l.o.g. we may assume that the arbitrary but fixed point $p\in \SSS$ has coordinates $(0,0,\ldots,0,-\pi/2+\delta)$ with suitable $\delta \in [0,\pi]$. Setting $\Theta = [-\pi/2,\pi/2]$, with the function
  \begin{align*}
    g : \Theta^{m-1} \to [0,1],~ \theta = (\theta_1,\ldots,\theta_{m-1})\mapsto \prod_{j=1}^{m-1}\cos^{m-j}\theta_j
  \end{align*}
  we have the spherical volume element 
  $g(\theta)\,d\theta\,d\phi$. Additionally defining
  \begin{align*}
    h(\theta) = \prod_{j=1}^{m-1}\cos\theta_j\,,
  \end{align*}
  we have that
  \begin{align*}
  \wF(p) &= \wF(\mu) + \frac{2\alpha}{v_m} \big(C_+(\delta) - C_-(\delta)\big) + \delta^2(1-\alpha) =: G(\delta)
  \end{align*}
  with the two ``crescent'' integrals
  \begin{align*}
    C_+(\delta)&=  \int_{\Theta^{m-1}} \limits d \theta \, g(\theta) \, \int_{-\delta}^{0} \limits  d\phi \, \wrho(\mu,q)^2 = \int_{\Theta^{m-1}} \limits d \theta \, g(\theta)\, \int^{\delta}_{0} \limits 
    \Big(\arccos\big(h(\theta)\,\sin\phi)\big)\Big)^2\,d\phi
    \\
    C_-(\delta)&= \int_{\Theta^{m-1}} \limits d \theta \, g(\theta) \, \int_{\pi-\delta}^{\pi} \limits  d\phi \, \wrho(\mu,q)^2= \int_{\Theta^{m-1}} \limits d \theta \, g(\theta) \, \int_{0}^\delta \limits \Big(\arccos\big(-h(\theta)\,\sin\phi)\big)\Big)^2\,d\phi
  \end{align*}
  cf. Figure \ref{Fig1}, because the spherical measure of $\LLL$ is $v_m/2$. 
  
  Since for $a\in [0,1]$, 
  \begin{align*}
    \big(\arccos (a\big))^2 - \big(\arccos(-a)\big)^2 &= \big(\arccos(a) + \arccos(-a)\big) \big(\arccos(a) - \arccos(-a)\big)\\
    &=2\pi \left(\frac{\pi}{2} -\arccos(a)\right) = -2\pi \arcsin(a)\,,
  \end{align*}   
  which has arbitrary derivatives if $-1<a<1$,
  we have that
  \begin{align}\label{eq:proof-lemma1}
    \wF\circ \exp_\mu(x)=G(\delta) = G(0) - \frac{4\pi\alpha}{v_m}  \int_{\Theta^{m-1}}\limits d \theta \, g(\theta)\,\int^{\delta}_{0} \limits  \arcsin\big(h(\theta)\,\sin\phi\big)\,d\phi + \delta^2(1-\alpha)
  \end{align}
  for every $x\in \exp_\mu^{-1}(\wU)$ with $\|x\| = \delta$,
  yielding the first assertion of the Lemma.
  
  For the second assertion we use the Taylor expansion
  \begin{align}\label{Taylor-arcsin:eq}
    \arcsin\big(h(\theta)\,\sin\phi\big) = \phi \,h(\theta) + \frac{\phi^3}{6} \big(h(\theta)^3) - h(\theta)\big) + O(\phi^5)
  \end{align}
  and compute for $k=0,1,\ldots$,
  \begin{align} 
    \int_{\Theta^{m-1}}\limits g(\theta)\,h(\theta)^k\,d\theta &= \int_{\Theta^{m-1}}\prod_{j=1}^{m-1}\left(\cos^{m-j+k}\theta_j \,d\theta_j\right) \nonumber\\
    &= \int_{\Theta^{m+k-1}}\prod_{j=1}^{m+k-1}\left(\cos^{m+k-j}\theta_j \,d\theta_j\right)\Big/ \int_{\Theta^{k}}\prod_{j=1}^{k}\left(\cos^{j}\theta_j \,d\theta_j\right) \nonumber\\
    &=\frac{v_{m+k}}{v_{k+1}} \label{eq:h2k}\,,
  \end{align}
  to obtain, in conjunction with \eqref{eq:proof-lemma1},
  \begin{align*}
    G(\delta) &= G(0) + \delta^2\left(1 - \alpha \left(1 + \frac{v_{m+1}}{2v_m}\right)\right) + \frac{\delta^4}{24} \,\frac{\alpha v_{m+1}}{v_m}\,\frac{m-1}{m+2} + \ldots
  \end{align*}
  which yields that for any choice of $\alpha \in [0,1]$ we have $G'(0) =0=G'''(0)$, as well as $G''(0) \geq 0$ for $1 \geq \alpha\left(1 +\gamma_m\right)$ with equality for $\alpha=1/\left(1 +\gamma_m\right)$. Since $G''''(0) =\frac{\alpha v_{m+1}}{v_m}\,\frac{m-1}{m+2}=c_m>0$ for all $\alpha \in (0,1)$, this guarantees a local minimum for $\alpha=1/(1 + \gamma_m)$. 
  
  In order to see that $\mu$ gives the global minimum in case of $\alpha = 1/\left(1 +\frac{v_{m+1}}{2v_m}\right)$ we consider the derivatives
  \begin{align}
    G'(\delta) &= - \frac{4\pi\alpha}{v_m}  \int_{\Theta^{m-1}} \limits  g(\theta) \,  \arcsin\big(h(\theta)\,\sin\delta\big) \, d \theta + 2\delta(1-\alpha)\,,\nonumber\\
    G''(\delta) &= - \frac{4\pi\alpha}{v_m}  \int_{\Theta^{m-1}} \limits  g(\theta)\,h(\theta) \,  \frac{\cos\delta}{\sqrt{1-h(\theta)^2\,\sin^2\delta}} \, d \theta + 2(1-\alpha) \label{eq:proof-lem-1}\\ 
    &\geq - \frac{4\pi\alpha}{v_m}  \int_{\Theta^{m-1}} \limits  g(\theta)\,h(\theta)   \, d \theta + 2(1-\alpha) = 2 - \alpha \left(2 + \frac{v_{m+1}}{v_m}\right) = 0\,, \nonumber
  \end{align}
  where the inequality is strict for $\delta \neq 0,\pi$, i.e. $p \neq \pm \mu$, due to $0<h(\theta)<1$ for  all $\theta \in (-\pi/2,\pi/2)^{m-1}$. Hence we infer that $G'(\delta)$ is strictly increasing in $\delta$ from $G'(0)=0$, yielding that there is no stationary point for $F$ other than $p=\mu$.
  \begin{figure}\centering
    \includegraphics[width=0.5\textwidth]{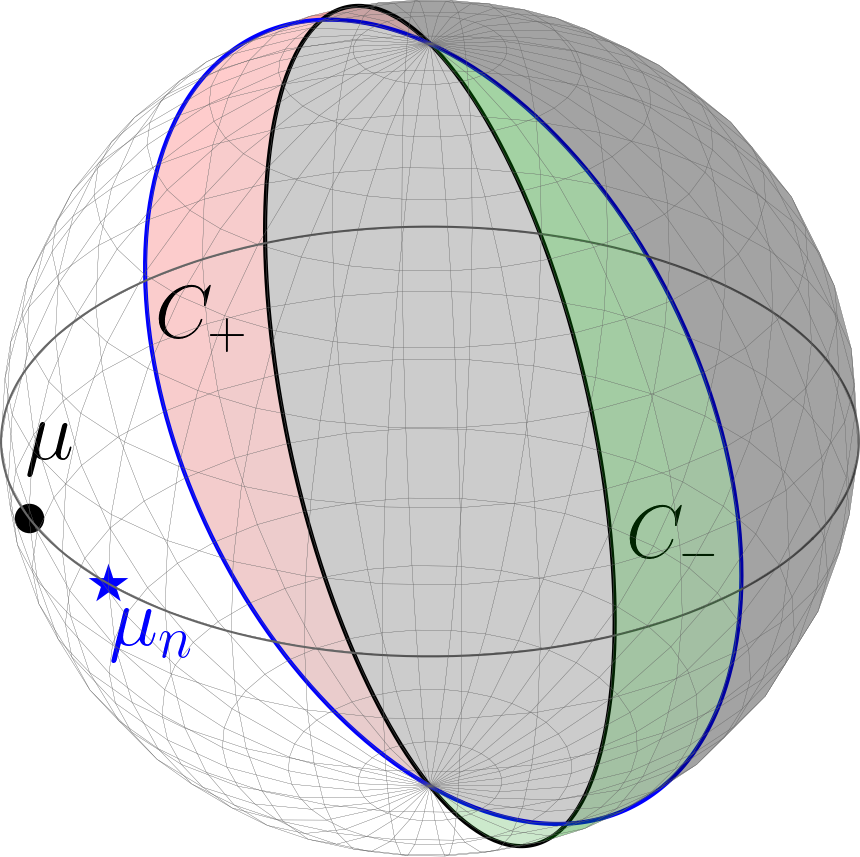}
    \caption{\it Depicting the two crescents $C_+=C_+(\delta)$ and $C_-=C_-(\delta)$ for $\delta=\arccos\langle \mu,\mu_n\rangle$ on $\SSS$ for $m=2$ with north pole $\mu$ and nearby sample Fr\'echet mean $\mu_n$ .\label{Fig1}}
  \end{figure}
\end{proof}

\begin{Rm}\label{Rmk1} ~\\\vspace*{-\baselineskip}
  \begin{enumerate}[(i)]
    \item Note that the result of  \citet[Proposition 3.1]{BL17} is not applicable to our setup as they have shown that on an arbitrary dimensional sphere the Fr\'echet function is twice differentiable, if the random direction has a density that is twice differentiable w.r.t. spherical measure. For the theorem to follow, we require fourth derivatives.
    \item Note that $O(\phi^5)$ in the Taylor expansion \eqref{Taylor-arcsin:eq} stands for 
    \begin{align*}
      \sum_{j=2}^\infty \phi^{2j+1} \sum_{r=0}^j a_{2r+1,2j+1} h(\theta)^{2r+1}
    \end{align*}
    with suitable coefficients $a_{2r+1,2j+1}\in \RR $. Moreover, due to \eqref{eq:h2k} we have
    \begin{align*}
      \frac{1}{v_m}  \int_{\Theta^{m-1}}\limits g(\theta)\,h(\theta)^{2r+1}\,d\theta &= \frac{1}{v_{2r+2}}\frac{v_{m+2r+1}}{v_m} \\
      & = \left\{\begin{array}{ll}
      \frac{1}{v_{2}}\frac{v_{m+1}}{v_m} &\mbox{ for }r=0\,,\\
      \frac{1}{v_{2}}\frac{v_{m+1}}{v_m} \prod_{k=1}^r \frac{2k}{m+2k-1}&\mbox{ for }r>0\,.
      \end{array}\right\}=O\left(\frac{1}{\sqrt{m}}\right)\,,
    \end{align*}
    due to Stirling's formula $\Gamma(z) = \sqrt{\frac{2\pi}{z}} \left(\frac{z}{e}\right)^z \left(1+O\left(\frac{1}{z}\right)\right)$.
    In consequence, cf. \eqref{eq:proof-lem-1}, $G^{(k)}(0) =0$ for odd $k\in \NN$ and $G^{(k)}(0) =O\left(\frac{1}{\sqrt{m}}\right)$ for even $4\leq k\in \NN$, as $m\to \infty$.
 \end{enumerate}
\end{Rm}

\subsection{A Two-Smeary Central Limit Theorem}\label{scn:spherical}

For a sample  $X_1,\ldots,X_n$ on $\SSS$ recall the \emph{empirical Fr\'echet function}
\begin{align*}
  \wF_n:\SSS \to [0,\infty),~p \mapsto \frac{1}{n}\sum_{j=1}^n \wrho(p,X_j)^2\,,
\end{align*}
where every minimizer $ \mu_n \in \SSS$ of $\wF_n$ is called an \emph{intrinsic Fr\'echet sample mean} or short just a \emph{sample mean}.

\begin{Th}\label{thm:smeary_sphere}
  Let $X_1,\ldots,X_n$ be a sample from $X$ as introduced in the setup Section \ref{setup:scn} with $\alpha = 1/(1+\gamma_m)$. Then, every measurable selection of sample means 
  \begin{align*}
  \mu_n&\in \argmin_{p\in \SSS} \wF_n(p) 
  \end{align*}
  is two-smeary. More precisely, with the exponential chart $\exp_\mu$ at the north pole, there is a full rank $m\times m$ matrix $\Sigma$ such that
  \begin{align*}
  \sqrt{n}\big(\exp^{-1}_\mu(\mu_n)\big)^3 \to \cN(0,\Sigma) 
  \end{align*}
  where the third power is taken component-wise. 
\end{Th}

\begin{proof}
  From Lemma \ref{lem:crescent} we infer that $\mu$ is the unique intrinsic Fr\'echet mean and hence by the strong law of \citet[Theorem 2.3]{BP03} we have that $\mu_n \to \mu$ almost surely yielding that Assumption \ref{as:unique} holds. Since $\SSS$ is an analytic Riemannian manifold also Assumption \ref{as:local-manifold} holds for arbitrary $r\in \NN$. With the exponential chart $\exp_\mu^{-1}:\wU \to \RR^m$ we have $\exp_\mu^{-1}(\mu) =0$ and we set $\exp_\mu^{-1}(\mu_n) = Z_n$ on $\{\mu_n \neq -\mu\}$ with $\Prb\{\mu_n \neq -\mu\} \to 1$ and $Z_n \inas 0$. 
  
  Further,  due to Lemma \ref{lem3}, the  family of functions
  \begin{align*}
    \rho(z,X) =  \wrho\big(\exp_\mu(z),X\big)^2,~z\in U
  \end{align*}
  has a.s. derivatives $\grad_z\rho(z,X)$, which are bounded, and on a compact set, are square integrable, so that Assumption \ref{as:Lipschitz} holds.
  
  Recalling the function $G(\delta)$ from the proof of Lemma \ref{lem:crescent} with its Taylor expansion, we have with $\delta = \wrho\big(\exp_\mu(z),\mu\big)=\|z\|$ that
  \begin{align*}
    \EE[g_z(X)] = G(\delta) = G(0) - \delta^4 \,\frac{c_m}{24} + \ldots
  \end{align*}
  and in consequence, Assumption \ref{as:Taylor-2nd} holds with  $r=4$, Thus, Theorem \ref{thm:CLT} is applicable.
  
  In particular, for the covariance we have
  \begin{align*}
    \Sigma &= \frac{36}{c^2_m} \cov\big[\grad_z|_{z=0}\,\wrho(\exp_\mu(z),X)^2\big]\,,
  \end{align*}
  which has full rank because in the exponential chart, rotational symmetry is preserved. This yields the assertion.

\end{proof}

\begin{Lem}\label{lem3}
  For $x \in \SSS$ and $z\in \RR^m\setminus\{\exp_\mu^{-1}(-x)\}$, $\|z\| < \pi$, 
  \begin{align*}\grad_z \Big(\wrho\big(\exp_\mu(z),x\big)\Big)\end{align*}
  is well defined  and has bounded directional limits as $z \to \exp_\mu^{-1}(-x)$ or $\|z\| \to \pi$.
\end{Lem}

\begin{proof}
  Recalling that $\wrho\big(\exp_\mu(z),x\big) = \acos \! \langle x,\exp_\mu(z)\rangle^2$, we have
  \begin{align}\label{lem3-eq}
  \grad_z \Big(\wrho\big(\exp_\mu(z),x\big)\Big)&= -2\,\frac{\grad_z \langle x,\exp_\mu(z)\rangle}{\sqrt{1- \langle x,\exp_\mu(z)\rangle^2}}~ \acos \! \langle x,\exp_\mu(z)\rangle \, .
  \end{align}
  In case of $x\neq 0$ this is bounded for $\|z\| \to \pi$. As we now show boundedness of \eqref{lem3-eq} also for $z\to \exp_\mu^{-1}(-x)$ for arbitrary $x \in \SSS$, also the boundedness in case of $x=0$ and $\|z\| \to \pi$ follows at once.
  
  To this end let $z$ be near $\exp_\mu^{-1}(-x)$ such that $z = \exp_\mu^{-1}(-x)+w$ with $w=(w_1,\ldots,w_m)\in \RR^m$ small. Then the asserted boundedness of (\ref{lem3-eq}) follows from
  \begin{align}\label{lem3-eq2}
    \langle x,\exp_\mu(z) \rangle = \langle x,\exp_\mu(\exp_\mu^{-1}(-x) + w) \rangle = -1 + w^T B w + \mathcal{O}(\|w\|^3)
  \end{align}
  with a symmetric matrix $B$, because then
  \begin{align*}
  \frac{\grad_z \langle x,\exp_\mu(z)\rangle}{\sqrt{1- \langle x,\exp_\mu(z)\rangle^2}} = \frac{2Bw + \mathcal{O}(w^2)}{\sqrt{2 w^T B w  + \mathcal{O}(w^3)}} \,,
  \end{align*}
  which is bounded for $w\to 0$ with any (possibly vanishing) symmetric $B$. 
  
  Finally, in order to see that the gradient  w.r.t. $w$ of the l.h.s. of \eqref{lem3-eq2} vanishes at $w=0$ , w.l.o.g. assume that $\exp_\mu^{-1}(x) = (\theta,0,\ldots,0)^T$ for some $\theta \in [0,\pi]$, such that $x =  (\sin\theta, \cos\theta, 0, \dots, 0)^T$ and $\exp_\mu^{-1}(-x) = (\pi-\theta,0,\ldots,0)^T$. Moreover, verify that
  \begin{align*}
  \exp_\mu(z) &= \left(\frac{-\pi+\theta+w_1}{\|z\|}\,\sin \|z\|,\cos\|z\|,*,\ldots,*\right)
  \end{align*}
  which, in conjunction with $\|z\| = \sqrt{ (\pi - \theta)^2 - 2 w_1 (\pi - \theta) + \|w\|^2}$, and hence,
  $\grad_w|_{w=0} \|z\| = - e_1$, yields that
  \begin{align*}
  \grad_w &\langle x,\exp_\mu(\exp_\mu^{-1}(-x) + w)\rangle\\
  &=\left( (-\pi+\theta+w_1)  \sin \theta\left(\frac{\cos \|z\|}{\|z\|} - \frac{\sin \|z\|}{\|z\|^2}\right) - \cos\theta \sin\|z\|\right) \grad_w \|z\| \\
  &\hspace*{8cm} + \sin\theta \frac{\sin \|z\|}{\|z\|} e_1\,,
  \end{align*}
  giving
  \begin{align*}
  &\grad_w|_{w=0} \langle x,\exp_\mu(\exp_\mu^{-1}(-x) + w)\rangle \\
  &= \Big(\sin\theta \, \big( \textnormal{sinc} (\pi - \theta) - \textnormal{sinc} (\pi - \theta) - \cos(\pi - \theta) \big) - \cos\theta \sin(\pi - \theta) \Big) e_1\\
  &= -\sin\pi \, e_1 = 0\,,
  \end{align*}
  which proves the claim (\ref{lem3-eq2}).
\end{proof}

\section{High Dimension Low Sample Size Effects Near Smeariness}\label{scn:high-dim}

We illustrate the relevance of our result by simulations of the variance $V=\wF(\mu)$ (the Fr\'echet function at the point mass at the north pole $\mu$) from the above in the setup Section \ref{setup:scn} introduced distributions parametrized in $\alpha =   \alpha_{\text{crit}} + \beta\in [0,1]$, on $\mathbb{S}^m$, for dimensions $m=2$, $10$ and $100$. Here the critical value $\alpha_{\text{crit}} = \frac{1}{1+\gamma_m}$ is $0.56$, $0.72$ and $0.89$, respectively.

We consider sample sizes ranging from $30$ to $10^6$ data points. For every sample size, we draw $1000$ samples, determine the spherical mean for each sample and then determine their empirical Fr\'echet function at $\mu$, i.e. the sum of squared distances of the means from the north pole. As we have a non-unique circular minimum of the Fr\'echet function for $\beta >0$, we expect in this case that the variance $V$ approaches a finite value, namely the squared radius of the circular mean set. For $\beta\leq 0$ we have a unique minimum, where for $\beta = 0$ we expect a slow decay of $V$ with rate approaching $n^{-\frac{1}{3}}$, due to Theorem \ref{thm:smeary_sphere}, and for $\beta < 0$ we expect the decay rate to approach $n^{-1}$.

\begin{figure}[ht!]
  \centering
  \subcaptionbox{$m=2$}[0.3\textwidth]{\includegraphics[height=0.3\textwidth]{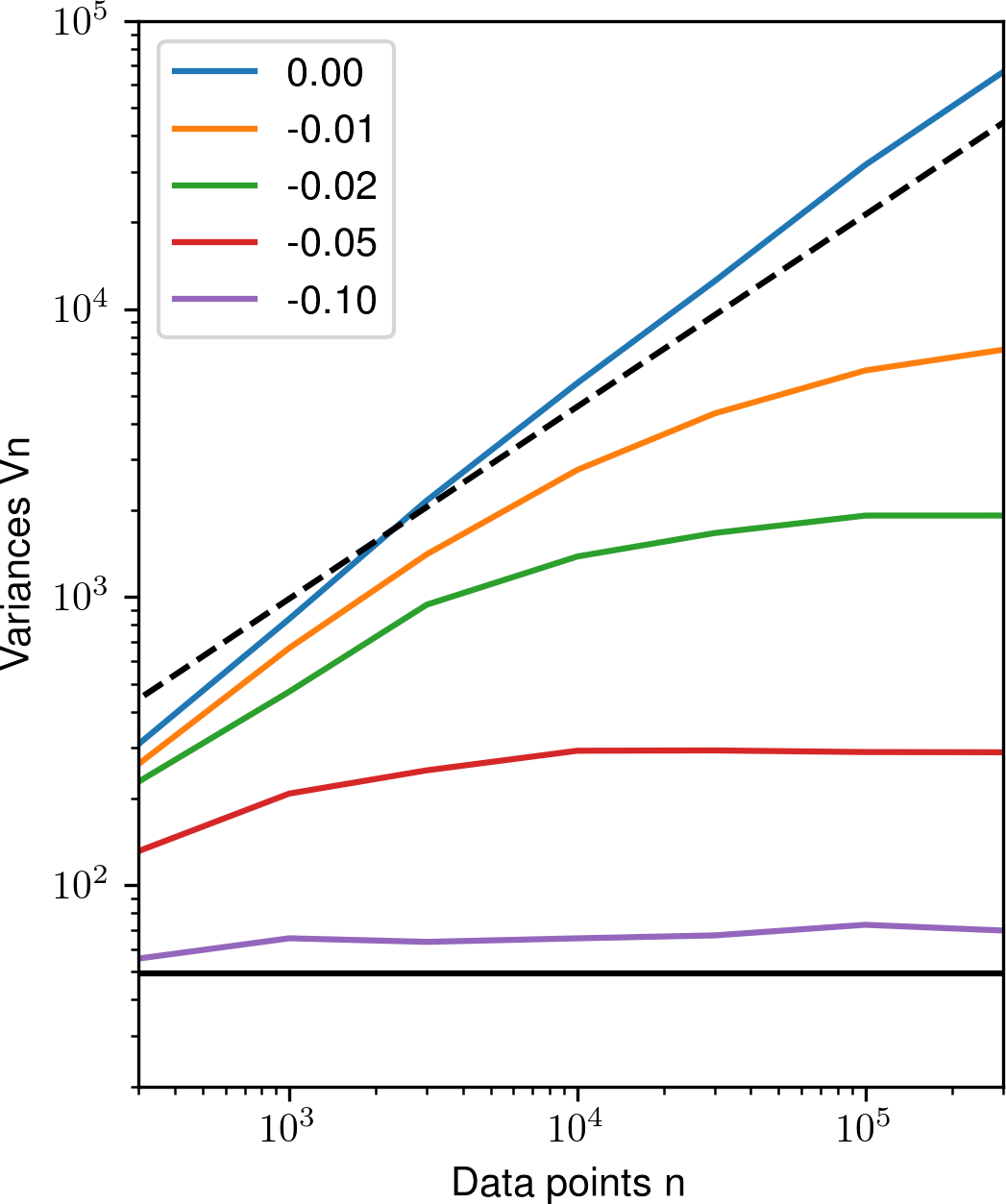}}
  \hspace*{0.02\textwidth}
  \subcaptionbox{$m=10$}[0.3\textwidth]{\includegraphics[height=0.3\textwidth]{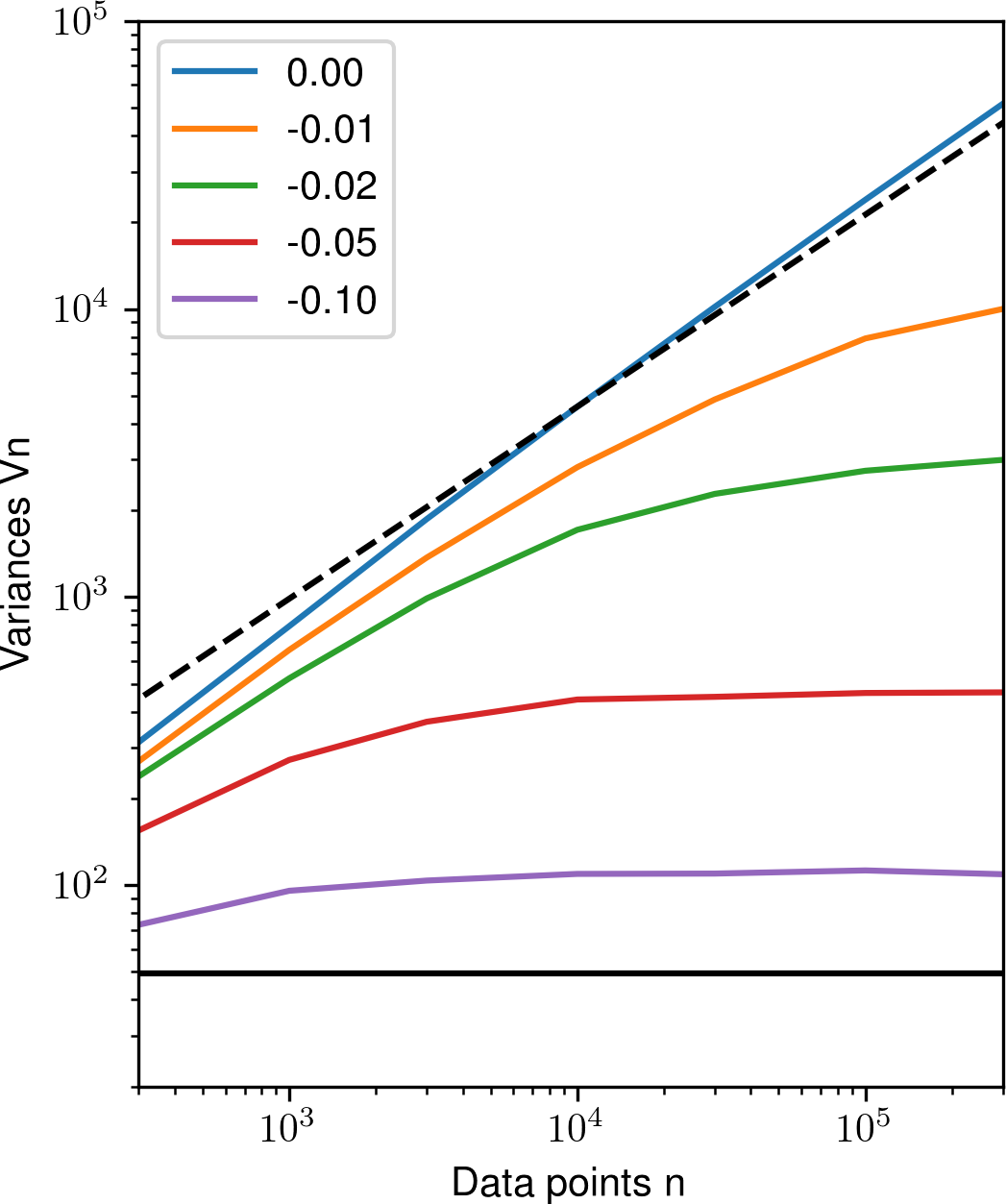}}
  \hspace*{0.02\textwidth}
  \subcaptionbox{$m=100$}[0.3\textwidth]{\includegraphics[height=0.3\textwidth]{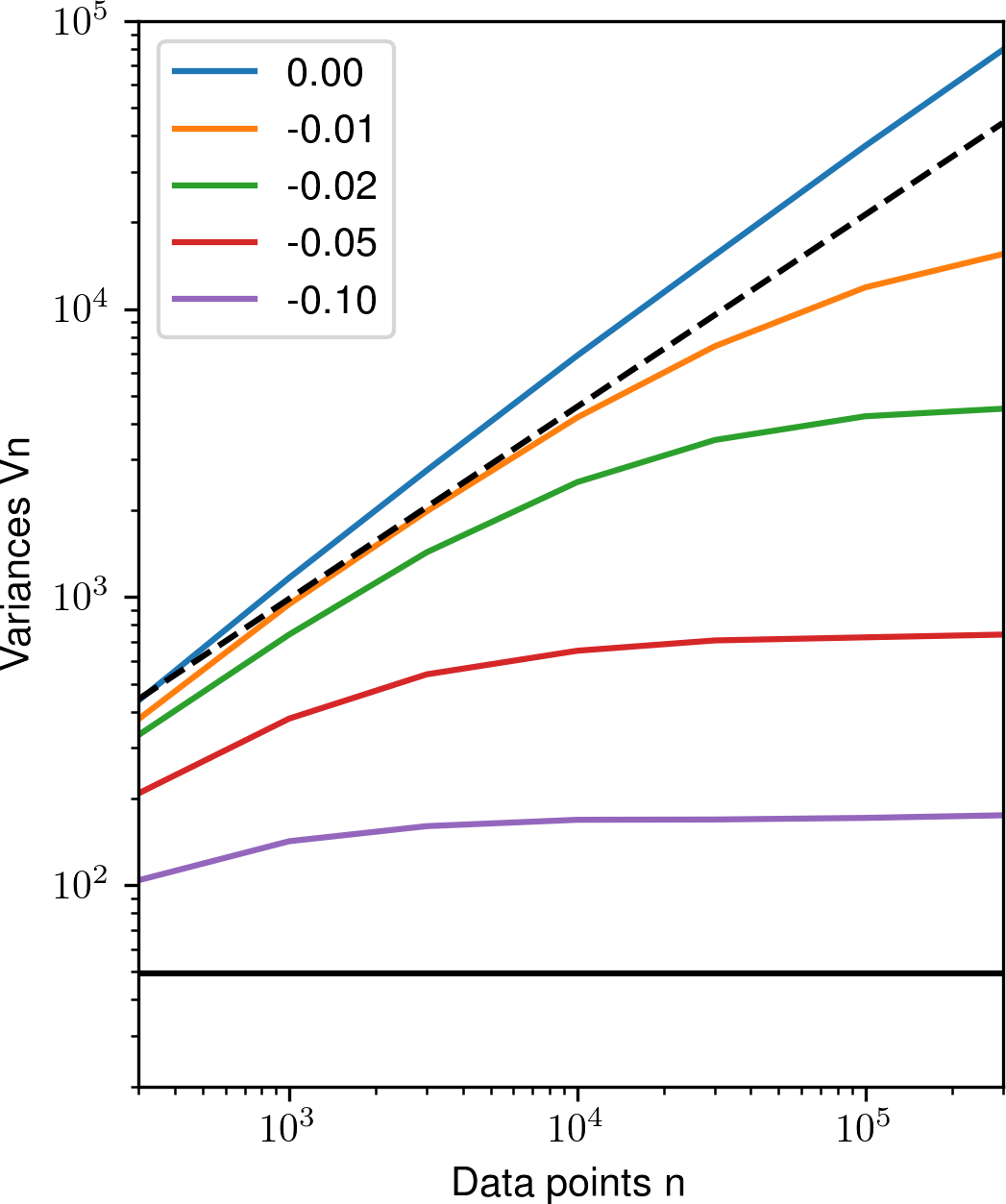}}
  \caption{\it Simulated variances $V$ times $n$ for different values of $\beta$ for dimensions $m=2$, $10$ and $100$. Black lines $V \propto n^{-1}$ (solid) and $V \propto n^{-\frac{1}{3}}$ (dashed) for reference.}
  \label{fig:sim_conv}
\end{figure}

The results of our simulation are displayed in Figure \ref{fig:sim_conv}. The asymptotic rates are clearly in agreement with our considerations based on the asymptotic theory. Strikingly, however, 
for $\beta < 0$ very close to $0$, the decay rate stays close to $n^{-\frac{1}{3}}$ until very large sample sizes and only then settles into the asymptotic rate of $n^{-1}$. This illustrates that the slow convergence to the mean is an issue, which does not only plague the distribution with $\beta = 0$ but also sufficiently adjacent distributions for finite sample size.

Even more strikingly, Figure  \ref{fig:sim_conv} shows that this phenomenon increases with dimension $m$. Indeed, due to Remark \ref{Rmk1} (ii), in the limit $m\to \infty$, all derivatives of $G$ vanish  with a uniform rate, so that we approach a situation of infinite smeariness.

\section*{Acknowledgment}

Both author gratefully acknowledge  DFG 755 B8, DFG HU 1575/4 and the Niedersachsen Vorab of the Volkswagen Foundation. We also thank Axel Munk for indicating to \cite{vanderVaar2000astat}. 

\bibliographystyle{Chicago}
\bibliography{superbib}

\end{document}